\documentclass[12pt,reqno]{amsart}

\usepackage[utf8]{inputenc}

\usepackage[normalem]{ulem}
\usepackage{ulem}
\usepackage[french,english]{babel}

\DeclareMathAlphabet{\mathpzc}{OT1}{pzc}{m}{it}

\usepackage{changepage, array}
\usepackage{xcolor}
\usepackage{graphicx}
\usepackage{comment}
\usepackage{hyperref}
\usepackage{enumerate}

\author{Hung Yean Loke}
\address{National University of Singapore, Science Drive 2, Singapore 117543}
\email{matlhy@nus.edu.sg}
\author{Tomasz Przebinda}
\address{Department of Mathematics, University of Oklahoma, Norman, OK 73019, USA}
\email{tprzebinda@ou.edu}

\title[The big theta]{The big theta}

\def\rSp{\mathrm{Sp}}
\def\tSp{\widetilde{\mathrm{Sp}}}
\def\tU{\widetilde{\mathrm{U}}}

\def\g{\mathfrak g}

\def\u{\mathfrak u}
\def\sp{\mathfrak {sp}}

\def\s{\mathfrak s}
\def\p{\mathfrak p}

\def\p{\mathfrak p}

\def\F{\mathbb{F}}

\def\Pp{\mathcal{P}}

\def\R{\mathbb{R}}
\def\C{\mathbb{C}}

\def\ss1{\mathfrak s_{\overline 1}}

\def\hs1{\mathfrak h_{\overline 1}}

\def\calH{\mathcal{H}}
\def\tG{\widetilde{G}}
\def\tK{\widetilde{K}}

\def\Ker{\mathrm{Ker}}

\def\G{\mathrm{G}}

\def\K{\mathrm{K}}

\def\Bbb{\mathbb}

\def\Sp{\mathrm{Sp}}

\def\rSp{\mathrm{Sp}}

\def\wt{\widetilde}

\def\Wv{\mathrm{W}}
\def\Vv{\mathrm{V}}
\def\Uv{\mathrm{U}}
\def\V{\mathsf{V}}

\def\End{\mathop{\hbox{\rm End}}\nolimits}

\def\Hom{\mathrm{Hom}}

\def\lim{\mathop{\hbox{\rm lim}}\nolimits}

\def\Z{\mathcal{Z}}

\def\Ss{\mathcal{S}}
\def\calS{\mathcal{S}}

\def\Nv{\mathrm{N}}

\def\fonttitre{\textsf}
\newcounter{thh}

\newtheorem{thm}[thh]{\fonttitre{Theorem}}

\newtheorem{pro}[thh]{\fonttitre{Proposition}}
\newtheorem*{pro*}{\fonttitre{Proposition}}

\newtheorem*{coro*}{\fonttitre{Corollary}}
\newtheorem{lem}[thh]{\fonttitre{Lemma}}

\newtheorem*{defi*}{\fonttitre{Definition}}

\newtheorem*{nota*}{\fonttitre{Notation}}

\theoremstyle{remark}

\def\muet{ \ifthenelse{\equal{a}{b}}}

\def\Z'{\Bbb{Z}'}

\def\biblio{\sloppy
\bibliographystyle{alpha}
\bibliography{article}}

\begin{document}

\date{}
\subjclass[2010]{Primary: 22E46; secondary: 22E47} 
\keywords{Howe's correspondence, reductive dual pairs over archimedean and non archimedean local fields.}

\date{\today}
\maketitle
\begin{abstract}
We show that Howe's big quotient is obtained via the tensoring over appropriate algebra. 
\end{abstract}

\def\v0{v_0}
\def\u0{u_0}
\def\sm{\mathrm {sum}}

\section{\bf Introduction}

Consider an irreducible dual pair $\tG\cdot \tG'$ in a metaplectic group $\tSp$ over a non-archimedean local field, \cite{MoViWa}. Assume that the members of the dual pair are of equal rank. Consider an irreducible tempered representation $\pi$ of $\tG$ and an irreducible tempered representation $\pi'$ of $\tG'$ in Howe's correspondence. In
\cite{MeslandSengun1} the authors show that $\pi'$ can be obtained from $\pi$ via tensoring the space of the smooth vectors of the Weil representation with $\pi$ over the reduced $C^*$ algebra of $\tG$. Here $\pi'$ happens to be the maximal Howe quotients (see \eqref{padichowemaximalquotient} below). 
In this note, we show that an analogous result holds in general, i.e. that the maximal Howe quotient (also known as the Big Theta) is always obtained via tensoring of the space of the smooth vectors of the Weil representation with the given representation over an appropriate (and obvious) algebra (Theorem \ref{padicbigtheta} below). 

We also verify the analogous result for the real reductive dual pairs (Theorem \ref{realhowebigquotient} below).

\section{\bf The $p$-adic big Theta}

\begin{pro} \label{Pdirectsum}
Let $\Ss$ be a complex vector space with a positive definite Hermitian form $( \, , \, )_{\Ss}$. Let $V$ be a finite dimensional subspace of $\Ss$. Then $\Ss = V \oplus V^\bot$ where $V^\bot$ is the orthogonal complement of $V$ in $\Ss$. Moreover the orthogonal project ${\mathrm{pr}}_V \colon \Ss \rightarrow V$ is a well defined linear transformation.
\end{pro}

\begin{proof}
It is clear that $\Ss \supseteq V \oplus V^\bot$.
It remains to show that for $w \in \Ss$, we have a decomposition $w = v + v'$ for some $v \in V$ and $v' \in V^\bot$.

Let $\{ v_1, \ldots, v_r \}$ be an orthonormal basis of $V$. Using the Gram-Schmidt process, we define
\[
v = (w, v_1)_{\Ss} v_1 +  (w, v_2)_{\Ss} v_2 + \cdots + (w, v_r)_{\Ss} v_r \in V.
\]
It is not hard to see that $(w - v, v_j)_{\Ss} = 0$ for $j = 1, \ldots, r$.
Hence $v' = w - v \in V^\bot$ and $w = v + v'$.
\end{proof}

Let $\F$ be a $p$-adic field. 
Let $\Wv$ be a finite dimensional symplectic vector space over $\F$.
Let $\tSp(\Wv)$ be the metalplectic double cover of $\rSp(\Wv)$.
Let $G \cdot G'$ be an irreducible dual pair $\rSp(\Wv)$, \cite[Lemma III.3]{MoViWa}.
Let $\tG$ and $\tG'$ be the inverse image of $\G$ and $\G'$ in the metalplectic double cover of $\tSp(\Wv)$.

\medskip

Fix a non-trivial unitary character $\chi$ of the additive group of the field $\F$ and
denote by  $\Ss$ be the space of the smooth vectors in the Weil representation of $\tSp(\Wv)$ associated to $\chi$. There is a positive definite $\tSp(\Wv)$-invariant inner product $(\ ,\ )_{\Ss}$ on $\Ss$.
Let $(\rho, V)$ be a genuine smooth irreducible representation of $\tG$.

\begin{lem} \label{L2}
Suppose $V\subseteq \Ss$  is a $\tG$ invariant subspace, such that the resulting representation of $\tG$ is smooth and irreducible. Then
\[
\calS = V \oplus V^\bot.
\]
\end{lem}

\begin{proof}
Let $w \in \calS$. Since $\Ss$ is a smooth representation, there is a compact open subgroup $L$ of $\tSp(W)$ such that $w$ is a fixed vector of $L$, i.e. $w \in \calS^L$. 
Let $K = L\cap \tG$. 
Since $V$ is an irreducible smooth representation of $\tG$, $V^K$ is finite dimensional, and is contained in $\calS^K$.
By applying Proposition \ref{Pdirectsum} to $V^K\subseteq \calS^K$ we see that $\calS^K = V^K \oplus (V^K)^\bot$ where $(V^K)^\bot$ is the orthogonal complement of $V^K$ in $\calS^K$.

Now $w \in \calS^K$, so $w = v + v'$ where $v \in V^K$ and $v' \in (V^K)^\bot$. Since $w$ and $v$ are $K$ - invariant, so is $v'=w-v$. Let $u \in V$. Then $\int_K \rho(k)u dk \in V^K$ and sine $(\ ,\ )_{\calS}$ is $\tG$-invariant,
\[
(u,v')_{\calS} = \int_K (\rho(k)u, \rho(k)v')_{\calS} dk = \int_K (\rho(k)u, v')_{\calS} dk
= (\int_K \rho(k)u dk, v')_{\calS} = 0.
\]
Thus $v' \in V^\bot$ and the lemma follows.

\end{proof}

\medskip

Denote by $\wt{\{1\}}\subseteq \tSp(\Wv)$ the preimage of the identity subgroup $\{1\}\subseteq \Sp(\Wv)$.
Let $\epsilon$ be the unique non-trivial character of the two element group $\wt{\{1\}}$.
Let $\calH(\tG)$ be the convolution algebra consisting of locally constant compactly supported complex valued functions on $\tG$ satisfying
\[
f(z g) = \epsilon(z) f(g)
\]
for $g \in \tG$ and $z \in \wt{\{1\}}$.


Let $(\rho, V)$ be a smooth, irreducible representation of $\tG$ which occurs as a quotient of $\Ss$ by a closed $\tG$ - invariant subspace.
Set
\[
\mathcal N_\rho=\bigcap_{T\in \Hom_{\tG}(\Ss, V)}\Ker T\,.
\]
Then, by \cite[Lemma III.3]{MoViWa} 
\begin{equation}\label{padichowemaximalquotient}
\calS/\mathcal N_\rho = V \otimes V'.
\end{equation}
for some smooth representation $(\rho', V')$ of $\tG'$.

Let $(\rho^\vee, V^\vee)$ denote the smooth dual representation of $(\rho,V)$. Since $V$ is a left $\calH(\tG)$-module, via $V\ni w\to \rho(f)w\in V$, $V^\vee$ is a right $\calH(\tG)$-module by defining for $v^\vee\in \V^\vee$ and $w\in V$, $(v^\vee f)(w)=v^\vee(\rho(f)w)$.

\begin{thm}\label{padicbigtheta}
In these terms,
\[
V'= V^{\vee} \otimes_{\calH(\tG)} \calS.
\]
\end{thm}
\begin{proof}
We have the following short exact sequence
\[
0\to\mathcal N_\rho\to \calS \to V \otimes V'\to 0\,.
\]
Hence the following sequence is exact:
\begin{equation}\label{3.1a}
V^{\vee}\otimes_{\calH(\tG)} \mathcal N_\rho \to V^{\vee} \otimes_{\calH(\tG)} \calS
\to V^{\vee} \otimes_{\calH(\tG)} V\otimes V'\to 0\,.
\end{equation}
By Schur's lemma,
\[
V^{\vee} \otimes_{\calH(\tG)} V \subseteq \Hom_{\tG}(V, V) =\C\,.
\]
This gives $V^\vee \otimes_{\calH(\tG)} V = \C$ and \eqref{3.1a} becomes
\begin{equation} \label{3.2}
V^{\vee}\otimes_{\calH(\tG)} \mathcal N_\rho \to V^{\vee} \otimes_{\calH(\tG)} \calS
\to V'\to 0\,.
\end{equation}
We will show later that
\begin{equation} \label{EqVeetimesN}
V^{\vee}\otimes_{\calH(\tG)} \mathcal N_\rho = 0.
\end{equation}
Then \eqref{3.2} gives $V^{\vee} \otimes_{\calH(\tG)} \calS
= V'$ and proves the theorem.

\medskip

We will prove \eqref{EqVeetimesN}. Suppose on the contrary, 
\[
V^{\vee}\otimes_{\calH(\tG)}\mathcal N_\rho\ne 0.
\]
Since $V^{\vee}\otimes_{\calH(\tG)}\mathcal N_\rho$ is contained in $\Hom_{\calH(\tG)}(\rho, \mathcal N_\rho)$, it means that $V \subseteq {\mathcal{N}}_\rho \subseteq \calS$.
By Lemma \ref{L2}
\[
\calS = V \oplus V^\perp.
\]
Thus $(\rho,V)$ is a quotient of $\calS$. This contradicts the inclusion $V \subseteq \mathcal N_\rho$. This proves~\eqref{EqVeetimesN}.
\end{proof}

\section{\bf The real Big Theta for Harish-Chandra modules}\label{Real Big Theta for Harish-Chandra modules}
Let $\Wv$ be a finite dimensional symplectic vector space over $\R$.
Let $\tSp(\Wv)$ be a metaplectic group with a maximal compact subgroup $\tU$.
Fix a non-trivial unitary character $\chi$ of the additive group of the field $\R$ of real numbers. 
Let $\Pp$ be  the Harish-Chandra module of 
the Weil representation of $\tSp(\Wv)$ associated to $\chi$. There is a positive definite $(\sp(\Wv),\tU)$-invariant inner product $(\ ,\ )_{\Pp}$ on $\Pp$. Consider an irreducible dual pair $\tG\cdot\tG'\subseteq \tSp$ and assume that $\tK=\tG\cap\tU$ is a maximal compact subgroup of $\tG$ and $\tK'=\tG'\cap\tU$ is a maximal comapct subgroup of $\tG'$.

Recall the convolution algebra $H(\g,\tilde K)$ of the $\tK$-finite distributions on $\tG$ supported in $\tK$, \cite{Knapp-Vogan}.  
Given an irreducible $(\g, \tilde \K)$ module
$(\rho, V)$, which occurs as a quotient of $\Pp$, set
\[
\mathcal N_\rho=\bigcap_{T\in \Hom_{(\g,\tK)}(\Pp, V)}\Ker T\,.
\]
Then there is a $(\g',\tK')$ module $(\rho', V')$, known as the big Howe quotient (or Big Theta), such that
\begin{equation}\label{realhowebigquotientdef}
\mathcal P/\mathcal N_\rho=V\otimes V'\,,
\end{equation}
see \cite[(2.5)]{HoweTrans}. 
For reader's convenience we shall verify this claim in the next section.
We shall view
\begin{equation}\label{isotypic_quotient}
V^{\vee}=\Hom(V,\C)_{\tilde K}
\end{equation}
as a right $H(\g,\tilde K)$ module.
\begin{thm}\label{realhowebigquotient}
In the above terms
\[
V'=V^{\vee}\otimes_{H(\g,\tilde K)}\mathcal P\,.
\]
\end{thm}
\begin{proof}
We have the following short exact sequence
\[
0\to\mathcal N_\rho\to \Pp\to V\otimes V'\to 0\,.
\]
Hence the following sequence is exact:
\begin{equation}\label{3.1}
V^{\vee}\otimes_{H(\g,\tilde K)}\mathcal N_\rho
\to V^{\vee}\otimes_{H(\g,\tilde K)}\Pp
\to V^{\vee}\otimes_{H(\g,\tilde K)}\V\otimes V'\to 0\,.
\end{equation}
By Schur's lemma,
\[
V^{\vee}\otimes_{H(\g,\tilde K)}V=\C\,.
\]
Suppose
\[
V^{\vee}\otimes_{H(\g,\tilde K)}\mathcal N_\rho\ne 0\,.
\]
Then
\[
\Hom_{H(\g,\tilde K)}(\rho, \mathcal N_\rho)\ne 0\,.
\]
Hence $V$ occurs as a submodule of $\Pp$. Now $\Pp$ is equipped with a positive definite $\s\p$ - invariant inner product $(\ ,\ )_{\Pp}$. Hence $(\ ,\ )_{\mathcal P}$ is $(\g,\tK)$ - invariant. Therefore
\begin{equation}\label{decomposition of P}
\mathcal P=V\oplus V^\perp
\end{equation}
is a $(\g,\tilde K)$ - invariant decomposition. Since the space $\mathcal P$ is infinite dimensional we need a few words about the existence of the decomposition \eqref{decomposition of P}. First of all we have the direct sum orthogonal decompositions
\[
\mathcal P=\bigoplus_{\pi\in\hat{\wt \K}} \mathcal P_\pi\,,\ \ \ V=
\bigoplus_{\pi\in\hat{\wt \K}} V\cap\mathcal P_\pi\,,
\]
where each $V\cap\mathcal P_\pi$ is finite dimensional. 

By Proposition \ref{Pdirectsum} above, ${\mathcal{P}} = V\cap\mathcal P_\pi \oplus \left(V\cap\mathcal P_\pi\right)^\perp$ so
\[
\mathcal P_\pi = (V\cap\mathcal P_\pi \oplus \left(V\cap\mathcal P_\pi\right)^\perp) \cap \mathcal P_\pi = V\cap \mathcal P_\pi \oplus ( \left(V\cap\mathcal P_\pi\right)^\perp \cap \mathcal P_\pi) \,.
\]
Therefore 
\[
\mathcal P = \bigoplus_{\pi\in\hat{\wt \K}} \mathcal P_\pi = \bigoplus_{\pi\in\hat{\wt \K}} 
\left(V\cap\mathcal P_\pi \oplus (\left(V\cap\mathcal P_\pi\right)^\perp\cap \mathcal P_\pi) \right)\,.
\]
Since the isotypic components corresponding to different $\wt K$ types are orthogonal, we see that
\[
\mathcal P=
\left(\bigoplus_{\pi\in\hat{\wt \K}} V\cap\mathcal P_\pi \right)\oplus 
\left(\bigoplus_{\sigma\in\hat{\wt \K}} (V\cap\mathcal P_\sigma)^\perp\cap \mathcal P_\sigma \right)=V\oplus V^\perp
\]
and \eqref{decomposition of P} follows.
Thus $V$ is a quotient of $\mathcal P$. This contradicts the inclusion
$
V\subseteq \mathcal N_\rho
$.
Therefore 
\[
V^{\vee}\otimes_{H(\g,\tilde K)}\mathcal N_\rho=0
\]
and \eqref{3.1} reduces to
\[
0\to V^{\vee}\otimes_{H(\g,\tilde K)}\mathcal P\to V'\to 0\,.
\]
\end{proof}
Let $\Ss$ be the space of the smooth vectors of the Weil representation, 
containing the Harish-Chandra module $\Pp$. Then
\[
V^{\vee}\otimes_{H(\g,\tilde K)}\Ss\supseteq V^{\vee}\otimes_{H(\g,\tilde K)}\mathcal P
\]
is a $\tG'$-module. It seems very likely that the Harish-Chandra module of the left hand side of the above inclusion is equal to the right hand side, hence the left hand side is a globalization of the right hand side. 
In addition, we believe that this is the smooth Howe maximal quotient, \cite[(1.1)]{HoweTrans}, corresponding to a globalization of $(\rho', V')$. Unfortunately we don't have any proof of either.

\section{\bf The isotypic quotient of a Harish-Chandra module}

In this section we copy \cite[Lemma III.3]{MoViWa}, with the proof, in terms of Harish-Chandra modules.

Let $\G$ be a real reductive group with a maximal compact subgraoup $\K\subseteq \G$ and the Lie algebra $\g$. Let $\Uv$ be a left $(\g,\K)$ module. Then 
\[
\Uv^c=\Hom(\Uv,\C)_\K
\]
is a left $(\g,\K)$ module. The map
\[
\Uv^c\times \Uv\ni (\psi, u)\to \psi(u)\in \C 
\]
extends to a non-zero element
\[
P\in \Hom_{(\g,\K)}(\Uv^c\otimes\Uv,\C).
\]
\begin{lem}\label{Schur-Dixmier}
If $\Uv$ is irreducible then
\[
\Hom_{(\g,\K)}(\Uv^c\otimes\Uv,\C)=\C P\,.
\]
\end{lem}
\begin{proof}
In terms of \cite[Proposition 2.47]{Knapp-Vogan}, 
\[
\Hom_{(\g,\K)}(\Uv^c\otimes\Uv,\C)=\Hom_{(\g,\K)}(\Uv^c,\Uv^c)\,.
\]
We see from Proposition 2.57 in the same book that $\Uv^c$ is also irreducible. Finally, by Schur-Dixmier lemma \cite[A.12]{Knapp-Vogan}
\[
\Hom_{(\g,\K)}(\Uv^c,\Uv^c)=\C\,.
\]
\end{proof}
Suppose $\Uv_0\subseteq\Uv$ is a $(\g,\K)$ submodule such that $(\g,\K)$ acts trivially on the quotient $\Uv/\Uv_0$. Let $\Uv_1\subseteq \Uv$ be the intersection of all the submodules $\Uv_0\subseteq \Uv$ with the above property. Then $\Uv/\Uv_1$ is a maximal trivial $(\g,\K)$ quotient of $\Uv$. Following \cite[III]{MoViWa} in the $p$-adic case, we shall denote it by 
\[
\Uv[\g,\K]\,.
\]
(A possible explanation for this notation is that $\Uv^{[\g,\K]}$ looks like the $(\g,\K)$ invariant subspace and $\Uv_{[\g,\K]}$
like an ``$(\g,\K)$ isotypic component". Hence $\Uv[\g,\K]$ is a reasonable compromise.)
\begin{lem}\label{Schur-Dixmier again}
If $\Uv$ is irreducible then
\[
(\Uv^c\otimes\Uv)[\g,\K]=\C\,.
\]
\end{lem}
\begin{proof}
Let
\[
S\in \Hom_{(\g,\K)}(\Uv^c\otimes\Uv,\C)
\]
be a nonzero map. We set $\Nv=\Ker S$. There is a nonzero constant $c$ such that $S=cP$. Hence $\Nv = \Ker S = \Ker P$. Thus $\Nv$ is unique and the claim follows.
\end{proof}
\begin{pro}\label{proposition 1}
Let $\Uv$ be a $(\g,\K)$ module and let $\Vv$ be an irreducible quotient of $\Uv$. Set
\[
\Vv ''=(\Vv^c\otimes \Uv)[\g,\K]\,.
\]
Assume that
\begin{equation}\label{intersection}
\bigcap_{T\in\Hom_{(\g,\K)}(\Uv, \Vv)}\Ker T=0\,.
\end{equation}
Then there is a subspace $\Vv '\subseteq \Vv ''$ such that $\Uv$ is isomorphic to $\Vv\otimes \Vv '$ as a $(\g,\K)$ module, with the trivial $(\g,\K)$ action on $\Vv'$.
\end{pro}

\noindent {\bf Remark.} The pair $(\g,\K)$ acts trivially on $\Vv''$ and $\Vv'$. The proposition says that $\Uv$ is isomorphic to a direct sum of possibly infinitely many copies of $\Vv$.
\begin{proof}
Let
\[
p \colon \Vv^c\otimes \Uv\to \Vv''
\]
be the projection (quotient map). Define a map 
\[
\phi \colon \Uv\to \Hom(\Vv^c,\Vv '')
\]
by
\[
\phi(u)(v^c)=p(v^c\otimes u)\,.
\]
A routine argument shows that $\phi$ intertwines the actions of $(\g,\K)$ on both sides. Let $T \colon \Uv\to\Vv$ be a $(\g,\K)$ intertwining map. Denote by
\[
p_1 \colon \Vv^c\otimes \Vv\to (\Vv^c\otimes \Vv)[\g,\K]
\]
the quotient map. The composition
\[
\Vv^c\otimes \Uv\overset{I\otimes T}{\rightarrow}\Vv^c\otimes \Vv
\overset{p_1}{\rightarrow}(\Vv^c\otimes \Vv)[\g,\K]
\]
is also $(\g,\K)$ intertwining. It factors through the projection $p$ and the result is a $(\g,\K)$ intertwining map
\[
T \colon \Vv''\to (\Vv^c\otimes \Vv)[\g,\K]\,.
\]
Let $0\ne u\in \Uv$. Then, by \eqref{intersection}, there is $T\in\Hom_{(\g,\K)}(\Uv,\Vv)$ and $v^c\in\Vv^c$ such that
\[
v^c(T(u))\ne 0\,.
\]
Hence,
\[
T'(p(v^c\otimes u))=p_1\circ(I\otimes T)(v^c\otimes u)=p_1(v^c\otimes u)=v^c(T(u))\,,
\]
where the last equality follows from Lemma \ref{Schur-Dixmier again}. Therefore
\[
\phi(u)(v^c) = p(v^c\otimes u) \neq 0\,.
\]
Hence $\phi\ne 0$. Even more, the above argument shows that $\phi$ is injective.

Fix an element $u$ in a $\K$ isotypic subspace of $\Uv$. Suppose this subspace is of type $\pi\in\hat\K$. Then $\chi_\pi=\dim\pi\cdot \Theta_\pi$ is the idempotent that realizes the projection onto that subspace.

Notice that for $k\in\K$,
\[
p(k\cdot v^c\otimes k\cdot u)=p(v^c\otimes u)\,.
\]
Hence
\[
p(v^c\otimes k\cdot u)=p(k^{-1}\cdot v^c\otimes u)\,.
\]
Therefore
\[
p(v^c\otimes u)=p(v^c\otimes \chi_\pi \cdot u)=p(\chi_{\pi^c} \cdot v^c\otimes u)\,.
\]
As a consequence,
\[
\phi(u)\in \Hom(\chi_{\pi^c}\cdot \Vv^c,\Vv '')=(\chi_\pi\Vv)\otimes \Vv ''\,,
\]
because $\dim\,(\chi_\pi\Vv)<\infty$. This identification allows us to see $\phi$ as a map with the range in $\Vv\otimes \Vv ''$. 
Altogether,
\[
\phi \colon \Uv\to \Vv\otimes \Vv ''
\]
is injective. Thus
\[
\phi(\Uv)\subseteq \Vv\otimes \Vv ''
\]
is a $(\g,\K)$ submodule. 

Pick a non-zero vector
\[
\sum_{i=1}^n v_i\otimes v_i ''\in \phi(\Uv)
\]
with $v_1$, ..., $v_n$ linearly independent and $v_1 ''$, ..., $v_n ''$ non-zero. Since $(\g,\K)$ acts irreducibely on $\Vv$,
\[
\End(\C v _1+...+\C v_n)
\]
is in the image of this action. Thus we may isolate each $v_i$ and map the other $v_j$ to zero, while staying in $\phi(\Uv)$. Therefore
\[
0\ne v_i\otimes v_i''\in\phi(\Uv)\,.
\]
Applying the $(\g, \K)$ action on $\Vv$ we see that
\[
\Vv\otimes v_i''\subseteq \phi(\Uv)\,.
\]
Hence,
\[
\Vv '=\{v''\in\Vv '';\ \Vv \otimes v''\subseteq \phi(\Uv)\}\ne 0
\]
and our argument shows that
$\phi(\Uv)\subseteq \Vv\otimes \Vv'$.
\end{proof}
The title of this section refers to the following statement.

\begin{thm}
Let $\G'$ be another real reductive group with a maximal compact subgroup $\K'\subseteq \G'$. Suppose $\Uv$ is a $(\g\oplus \g', \K\times \K')$ module such that the actions of $(\g, \K)$  and $(\g', \K')$  commute. Let $\V$ be an irreducible $(\g,\K)$ quotient of $\Uv$. Set
\[
\Nv=\bigcap_{T\in \Hom_{(\g,\K)}(\Uv, \Vv)}\Ker T\,.
\]
Then there is a $(\g', \K')$ module $\Vv_1'$ such that $\Uv/\Nv$ is isomorphic to $\Vv\otimes \Vv_1'$ as a $(\g, \K)\times(\g', \K')$ module:
\[
\Uv/\Vv=\Vv\otimes \Vv_1'\,.
\]
\end{thm}
\begin{proof}
The intersection $\Nv$ is a $(\g\oplus \g', \K\times \K')$ submodule of $\Uv$. Hence, $\Uv/\Nv$ is a $(\g, \K)\times(\g',\K')$ module on which the two actions commute. Also, $\Vv$ is an irreducible $(\g, \K)$ quotient of $\Uv/\Nv$. Furthermore
\[
\bigcap_{T\in \Hom_{(\g,\K)}(\Uv/\Nv,\Vv)}\Ker T=0\,.
\]
Hence Proposition \ref{proposition 1} applies. The space $\Vv_1'$ constructed in the proof of that proposition is a $(\g',\K')$ module and as a result
\[
\Uv/\Nv=\Vv\otimes \Vv_1'
\]
as a $(\g, \K)\times(\g',\K')$ module.
\end{proof}

\biblio
\end{document}